\documentclass[a4paper,10pt]{amsart}
\usepackage[all]{xy}
\usepackage{latexsym}
\usepackage{amssymb}
\usepackage{amsmath}

\newtheorem{theorem}{Theorem}[section]
\newtheorem{proposition}[theorem]{Proposition}
\newtheorem{corollary}[theorem]{Corollary}
\newtheorem{lemma}[theorem]{Lemma}
\newtheorem{example}[theorem]{Example}

\newtheorem{definition}[theorem]{Definition}
\newtheorem{fact}[theorem]{Fact}

\newtheorem{remark}[theorem]{Remark}
\newtheorem*{prob}{Main Problem}
\newtheorem*{mtheo}{Main Theorem}
\newtheorem*{theo}{Theorem}

\newcommand{\Ab}{{\rm Ab}}

\newcommand{\pp}{{\rm pp}}

\newcommand{\Th}{{\rm Th}}

\newcommand{\grO}{\Omega}

\newcommand{\gro}{\omega}

\newcommand{\grf}{\varphi}

\newcommand{\rbar}{\overline{r}}

\newcommand{\ba}{{\bf a}}
\newcommand{\bb}{{\bf b}}

\newcommand{\bm}{ {\bf m}}
\newcommand{\bv}{ {\bf v}}
\newcommand{\bu}{ {\bf u}}

\newcommand{\bx}{{\bf x}}

\newcommand{\bbN}{\mathbb{bN}}

\newcommand{\bbZ}{\mathbb{bZ}}

\begin{document}

\footskip30pt

\date{}

\title{A countable universal torsion abelian group for purity}
\author{Ivo Herzog}
\email{herzog.23@osu.edu}
\urladdr{https://u.osu.edu/herzog.23/}
\address{Department of Mathematics, The Ohio State University, Lima, Ohio, USA}

\author{Marcos Mazari-Armida}
\email{Marcos.MazariA@colorado.edu}
\urladdr{https://math.colorado.edu/~mama9382/}
\address{Department of Mathematics, University of Colorado Boulder, Boulder, Colorado, USA}

\subjclass[1991]{}
{\let\thefootnote\relax\footnote{{AMS 2020 Subject Classification: Primary: 20K30. Secondary: 03C45, 03C48, 03C60, 13L05.
Key words and phrases. Abelian groups, $p$-groups, purely universal models, $\aleph_0$-strongly homogeneous models.\\
The second author's research was partially supported by an AMS-Simons Travel Grant 2022-2024. }}}  

\begin{abstract}
 We show that there is a countable universal abelian $p$-group for purity, i.e., a countable abelian $p$-group $U$ such that every countable abelian $p$-group purely embeds in $U$. This is the last result needed to provide a complete solution to Problem 5.1 of \cite{fuchs} below $\aleph_\omega$. We introduce $\aleph_0$-strongly homogeneous $p$-groups, show that there is a universal abelian $p$-group for purity which is $\aleph_0$-strongly homogeneous, and completely characterize the countable $\aleph_0$-strongly homogeneous $p$-groups.

\end{abstract}

\maketitle

\section{Introduction}

The main goal of this paper is to address Problem 5.1 in page 181 of \cite{fuchs} in the case of countable abelian $p$-groups.  The problem stated by Fuchs is the following:

\begin{prob}\label{mprob}
For which cardinals $\lambda$ is there a universal abelian $p$-group for purity? We mean an abelian $p$-group $U_\lambda$ of cardinality $\lambda$ such that every abelian $p$-group of cardinality $\leq \lambda$ embeds in $U_\lambda$ as a pure subgroup. The same question for torsion-free abelian groups.
\end{prob}

We provide a positive answer in the case $\lambda = \aleph_0$ for $p$-groups.\footnote{This also answers Question 3.11 of \cite{m4}.} More precisely we obtain the following result. Recall that $H$ is a pure subgroup $G$ if $n G \cap H = n H$ for every $n \in \mathbb{N}$.

\begin{mtheo}
$\bigoplus_{n =1}^\infty \bbZ (p^n)^{(\aleph_0)} \oplus \bbZ (p^\infty)^{(\aleph_0)}$ is a countable universal abelian $p$-group for purity.
\end{mtheo}

The above result is the last result needed to provide a complete solution to Fuchs' problem below $\aleph_\omega$. For \emph{$p$-groups}, it is shown in \cite{m4}, that the existence of a universal abelian $p$-group for purity of cardinality $\aleph_1$ is independent of the standard axioms of set theory (ZFC) and that the existence of a universal abelian $p$-group for purity of cardinality $\aleph_n$ depends on the value of the continuum for $n \geq 2$. For \emph{torsion-free abelian groups}, it is shown in  \cite{kojsh}, that there is not a universal torsion-free abelian group for purity of cardinality $\aleph_0$ and that the existence of a universal torsion-free abelian group for purity of cardinality $\aleph_n$ depends on the value of the continuum for $n \geq 1$. There are some partial results for cardinals greater than or equal to $\aleph_\omega$  for $p$-groups in \cite{m4} and significant results for torsion-free groups in \cite{kojsh}, \cite{sh3}, \cite{sh820}, \cite{kuma}, \cite{sh1151}. On page 302 of \cite{sh1151} a detailed breakdown for cardinals greater than $\aleph_\omega$ is presented.

We give two independent arguments for the existence of a countable $p$-group for purity. The first argument, which is presented in Section 3, uses only standard model theoretic tools. The key result is that the number of $p$-torsion pp-types of any finite length in the theory of abelian groups is countable (Theorem \ref{T:p-small}). Once that is shown, we take the \emph{biggest} complete theory of abelian $p$-groups and construct a countable $\aleph_0$-saturated model for $p$-torsion pp-types which is shown to be a universal $p$-group for purity.

The second argument, which is presented in Section 4, uses some deep results of the model theory of modules. The advantage of this argument is that we explicitly describe a countable universal $p$-group for purity and show that for every countable $p$-group there is a pure essential monomorphism into a direct summand of the mentioned universal group. The key result is that the pure injective hull of a finite subgroup of a $p$-group is the direct sum of a finite number of indecomposable pure injective $p$-groups (Theorem \ref{T:main}). Once that is shown, we build a pure essential monomorphism from any countable $p$-group into a direct summand of $\bigoplus_{n =1}^\infty \bbZ (p^n)^{(\aleph_0)} \oplus \bbZ (p^\infty)^{(\aleph_0)}$ by decomposing  the $p$-group into finite subgroups and using the universal property of the pure injective hull. 

Then we turn to study the classical model theoretic notion of an $\aleph_0$-strongly homogeneous model\footnote{Also called $\aleph_0$-sequentially strongly homogeneous model in the model theory literature.} for $p$-groups, i.e., a $p$-group such that every partial pure isomorphism (see Definition \ref{D: partial}) of finite subgroups extends to an automorphism of the group. We show that direct sums of pure injective indecomposable $p$-groups are $\aleph_0$-strongly homogeneous, so in particular the universal model  $\bigoplus_{n =1}^\infty \bbZ (p^n)^{(\aleph_0)} \oplus \bbZ (p^\infty)^{(\aleph_0)}$ is $\aleph_0$-strongly homogeneous. Moreover, we get a complete characterization of the countable $\aleph_0$-strongly homogeneous $p$-groups.

\begin{theo} Assume $M$ is a countable $p$-group. The following are equivalent.
\begin{enumerate}
\item $M$ is $\aleph_0$-strongly homogeneous.
\item One of the following holds.
\begin{enumerate}
\item $M$ is isomorphic to $ \bigoplus_{n =1}^\infty \bbZ (p^n)^{(\alpha_n)} \oplus \bbZ (p^\infty)^{(\alpha_\infty)}$ such that $\alpha_n$ is countable (possibly finite) for every $n \in \mathbb{Z}_{> 0} \cup \{\infty\}$. 
\item There is an $n \geq 1$ such that $p^{\omega + n} M =0$ and $p^\omega M[p] = p^{\omega + (n -1)} M[p] \neq 0$ where $p^{\omega  + i}M[p]$ is the subgroup of elements annihilated by $p$ in $p^{\omega  + i}M$ for every $i \in \mathbb{N}$. 
\end{enumerate}
\end{enumerate}
\end{theo}

For an uncountable cardinal $\lambda$, we show that there are $2^\lambda$ non-isomorphic   $\aleph_0$-strongly homogeneous $p$-groups of cardinality $\lambda$ by showing that $p$-groups without elements of infinite height are $\aleph_0$-strongly homogeneous and using \cite{sh74}. We finish the paper by showing how $\aleph_0$-strongly homogeneous $p$-groups relate to the classical notion of a transitive group of Kaplansky \cite[\S 18]{kapla}.


The paper is divided into five sections. Section 2 presents some preliminaries. Sections 3 and 4 have the model theoretic argument and algebraic argument of the existence of a countable $p$-group for purity respectively. Section 5 studies $\aleph_0$-strongly homogeneous $p$-groups.

We would like to thank John T. Baldwin, Samson Leung, Philipp Rothmaler and Wentao Yang for comments that helped improve the paper. We are grateful to the referee for many comments that significantly improved the presentation of the paper.

\section{Preliminaries} 

We introduce the key concepts of the model theory of modules and abelian group theory that are used in this paper. These are further studied in \cite{prest}, \cite{ziegler} and \cite{fuchs}.

\subsection{Positive primitive formulae} Since we will only study classes of abelian groups, the language will always be $L_\mathbb{Z} = \{0, +,-\} \cup \{ z\cdot   : z \in \mathbb{Z} \}$ where $z \cdot$ is interpreted as multiplication by $z$ for every $z \in \mathbb{Z}$.  We will write \emph{group} instead of \emph{abelian group}. The axioms for the class of abelian groups is expressible in $L_\mathbb{Z}.$ The theory of abelian groups is denoted by $\Th (\Ab).$ 

\begin{definition} A formula $\grf(\bv)$ is a \emph{positive primitive formula}, pp-formula for short,  provided it is equivalent, relative to $\Th (\Ab),$ to a formula of the form
\begin{equation} 
\grf(\bv )= \exists w_1,  \ldots, \exists w_l \; \bigwedge_{j=1}^{m} \; \left(
\sum_{i=1}^{n} z_{i,j} v_i + \sum_{k=1}^{l} s_{k,j}w_{k} \doteq 0 \right) ,
\end{equation}
where $z_{i,j}, s_{k,j} \in \mathbb{Z}$ for every $i \in \{ 1,..., n\}, j \in \{1,...,m\}, k\in \{1,...,l\}$ and \hspace{0.01cm} $\doteq$ is the formal symbol for equality. 
\end{definition}

It follows from The Fundamental Theorem of Finitely Generated Modules over a PID, see for example \cite[2.$\mathbb{Z}$.1]{prest}, that every pp-formula in $L_\mathbb{Z}$ is equivalent, relative to $\Th (\Ab)$,  to one of the simplified form
\begin{eqnarray}
& \bigwedge_{j=1}^k \;   \left( s'_j \; \mid \; \sum_{i=1}^{n} z'_{i,j} v_j  \right) \label{Eq:conj}
\end{eqnarray}
where the abbreviation $s \mid v$ is used for the pp-formula $\exists w \; (sw \doteq v).$ This includes the possibility that $s = 0,$ which refers to the formula $0 \mid v$ which is equivalent to $0 \doteq v.$ We write $s \nmid v$ for $\neg (s \mid v)$.

\subsection{Subsets defined by pp-formulae} If $M$ is a group a pp-formula $\grf (\bv )$ defines in $M$ the subgroup 
\begin{equation}
\grf (M) \colon = \{ \; \ba \in  M^n \; \mid \; M \models \grf (\ba) \; \}
\end{equation}
of $M^n,$ where $n = \ell (\bv )$ is the length of the tuple of free variables. The definition $M \mapsto \grf (M)$ is natural (functorial) in the sense that if 
$f \colon M \to N$ is a homomorphism of groups, then so is $\grf (f) = f^n \! \mid_{\grf (M)} \colon \grf (M) \to \grf (N),$ the restriction to $\grf (M)$ of the diagonal map $f^n \colon M^n \to N^n$ induced by $f.$

We will mostly be interested in the subsets of $M$ defined by pp-formulae $\grf (v, \bb)$ in one free variable with parameters $\bb = (b_1, \ldots, b_n),$ $b_i \in M.$ If $\grf (v, \bb)$ is $M$-consistent, then 
$$\grf (M, \bb) := \{ \; a \in M \; \mid \; M \models \grf (a, \bb) \; \}$$
is a coset in $M$ of $\grf (M, {\bf 0}) \leq M,$ the subgroup defined in $M$ by the pp-formula $\grf (v, {\bf 0})$ without parameters.

\subsection{Partial homomorphisms} From now on, denote by $(M,B)$ an abelian group $M$ together with a distinguished subset $B$ contained in $M$.\footnote{For most of the paper one can assume that $B$ is a subgroup of $M$, this is the case in the algebraic argument of the main result of the paper.} A morphism $(f, f_0) \colon (M,B) \to (N,C)$ of \emph{groups with parameters} is just a morphism $f \colon M \to N$ of abelian groups such that $f(B) \subseteq C$ and $f_0 \colon B \to C$ is given by $f_0 = f \upharpoonright_B,$
$$\xymatrix{M \ar[r]^f & N \\
B \ar[r]^{f_0} \ar@{^{(}->}[u] & C. \ar@{^{(}->}[u]
}$$ 

The \emph{pp-type} of $(M,B)$ over $A$ a subset of $M$ is the collection of pp-formulae given by
$$\pp^+ (B/A,M) := \{ \; \grf (\bu_{\bb}, \ba ) \; \mid \; M \models \grf (\bb, \ba) \text{ and } \ba \in A \; \}.$$
The free variables of the pp-fomula $\grf (\bu_{\bb}, \ba)$ are indexed by the elements of $B.$ This is precisely the set of pp-formulae satisfied by $B$ in $M$ with parameters in $A$. For a tuple $\bb \in M^n$ we will write $\pp^+ (\bb /A,M)$. We will often consider the case where $A = \emptyset$ in which case we will simply write $\pp^+ (B,M)$.  

If $(f,f_0) \colon (M,B) \to (N,C)$ is a homomorphism of groups with parameters, let
$$f (\pp^+ (B,M)) := \{ \grf (\bu_{f(\bb)}) \; \mid \; \grf (\bu_{\bb}) \in \pp^+ (B,M) \; \}$$
be the collection of pp-formulae obtained by replacing the variables $u_b$ with $u_{f(b)}$ in the pp-formulae of $\pp^+ (B,M).$ The discussion above implies that $f (\pp^+ (B,M)) \subseteq \pp^+ (C,N),$ which suggests the following definition, due to Ziegler~\cite[\S 3]{ziegler}.

\begin{definition} \label{D:phomo} 
A \emph{partial homomorphism} $f_0 \colon (M,B) \to (N,C)$ between groups with parameters is a function $f_0 \colon B \to C$ of the parameter subsets such that 
$f_0 (\pp^+ (B,M)) \subseteq \pp^+ (C,N).$ In other words, for every pp-formula $\grf (\bu_{\bb})$ with free variables indexed by $B,$ 
$$M \models \grf (\bb) \mbox{ implies } N \models \grf (f_0(\bb)).$$
\end{definition}

It is clear that $1_B \colon (M,B) \to (M,B)$ is a partial homomorphism and that a composition of partial homomorphisms is a partial homomorphism. Definition~\ref{D:phomo} is similar to that of Ziegler~\cite[\S 3]{ziegler} but we specify a codomain of parameters. It also differs from the definition~\cite[Definition 2.1]{cortes} of Cortes Izurdiaga et.\ al., who define a partial homomorphism to be the pair of maps given by the inclusion $B \leq M$ together with the composition $B \stackrel{f_0}{\to} C \leq N.$

\begin{definition}\label{D: partial}
A partial homomorphism $f_0 \colon (M,B) \to (N,C)$ is a \emph{partial pure monomorphism} if $$f_0 (\pp^+ (B, M)) = \pp^+ (f_0 (B), N),$$ 
that is, if for every pp-formula $\grf (\bu_{\bb})$ with free variables indexed by $B,$ $$M \models \grf (\bb) \mbox{ if and only if } N \models \grf (f_0(\bb)).$$

It is a \emph{partial pure isomorphism} if it is a partial pure monomorphism and $f_0[B] = C$.
\end{definition}

This notion is similar to Ziegler's ~\cite[\S 3]{ziegler} \emph{partial isomorphism.} It is obvious by considering the pp-formula $u_b \doteq 0$ that if $f_0 \colon (M,B) \to (N,C)$ is a partial pure monomorphism and $B, C$ are groups, then $f_0 \colon B \to C$ is a monomorphism of groups. The converse is not true in general. Neither is a partial pure monomorphism necessarily a monomorphism in the category of groups with parameters (and morphisms the partial homomorphisms).

A homomorphism $f \colon M \to N$ of groups is a \emph{pure monomorphism} if the partial homomorphism \linebreak $f \colon (M,M) \to (N, N)$ is a partial pure monomorphism. If $M \leq N$ and the inclusion morphism is a pure monomorphism, we say that $M$ is a \emph{pure subgroup} of $N$ and we denote it by $M \leq_p N$. The Fundamental Theorem of Finitely Generated Modules over a PID may be applied again to show that this is equivalent to Pr\"{u}fer's original definition that for every $z \in \bbZ,$ $zM = zN \cap M.$

A morphism $f: M \to N$ is an elementary monomorphism, if for every first-order formula $\psi (\bx_{\bm})$ with free variables indexed by $M,$
$$M \models \psi (\bm) \mbox{ if and only if } N \models \psi (f(\bm)).$$
If $M \leq N$ and  the inclusion morphism is an elementary monomorphism, we say that $M$ is an \emph{elementary subgroup} of $N$ and we denote it by $M \preceq N.$ An elementary monomorphism $f \colon M \to N$ is necessarily a pure monomorphism, and the converse holds if and only if $M$ and $N$ are elementarily equivalent abelian groups \cite[2.25]{prest}.


\subsection{Torsion subgroups} If $M$ is a group,  the \emph{torsion subgroup of $M$} consists of the elements of finite order in $M$, we denote it by $t(M)$, and the \emph{$p$-primary component of $M,$} $p$ prime, consists of the elements whose order is a power of $p$ in $M$, we denote it by $t_p(M)$. $t(M)$ is a pure subgroup of $M$, for the general reason that the torsion-free quotient group $M/t(M)$ is flat. Furthermore, every torsion group $M$ may be decomposed as a direct sum $M = \bigoplus_p \; t_p(M),$ indexed by the primes $p$.
It follows that if $M$ is a general group and $p$ is a prime then the $p$-primary component $t_p(M)$ is a pure subgroup of $M$. An abelian group $M$ is a $p$\emph{-group} if $t_p(M) = M,$ that is, if every element has order a power of $p.$ The \emph{socle} of an abelian $p$-group consists of the subgroup of elements annihilated by $p$. 

\begin{proposition}
If $(M,A)$ is a group with a $p$-group of parameters $A,$ then $A \leq t_p(M),$ and \linebreak $1_A \colon (t_p(M), A) \to (M,A)$ is a partial pure isomorphism.
\end{proposition}

\begin{proof}
Since $t_p(M) \leq_{p} M$ the pp-type of $A$ in $M$ is the same as that of $A$ in $t_p(M).$
\end{proof}

In a $p$-group, every pp-formula $s \mid v$ for $s \in \mathbb{Z}$ is equivalent to the pp-formula $p^k \mid v$ where $k \in \mathbb{N}$ such that $s = p^k m$ and $(p, m)=1$. 

\subsection{Universal models}

Let us recall the classical notion of a universal model for a class of abelian groups.

\begin{definition}
Let $\mathbf{K} = (K ,\leq_K)$ where $K$ is a class of abelian groups and $\leq_K$ is a partial order on $K$. Let $\lambda$ be a cardinal. $M \in K$ is a \emph{universal model} in $\mathbf{K}$ of cardinality $\lambda$ if $M$ has cardinality $\lambda$ and if for any $N \in K$ of cardinality $\lambda$ there is $f: N \to M$ a group monomorphism such that $f[N] \leq_K M$. 
\end{definition}

When $K$ is the class of $p$-groups, $\leq_K \; = \; \leq_p$ and $\lambda = \aleph_0$; we will say that \emph{$M$ is a countable universal $p$-group for purity} instead of \emph{$M$ is a universal model in the class of $p$-groups of cardinality $\aleph_0$ with pure embeddings.}

\section{Model theoretic argument}

In this section we present a direct proof of the existence of a countable  universal $p$-group for purity. The results of this section are slightly weaker than those of the next section, but have the advantage of only using standard model theoretic tools and very basic results of the model theory of modules.

The class of $p$-groups is not first-order axiomatizable in the language of abelian groups $L_\mathbb{Z}$. This can be shown using the Compactness Theorem by an analogous argument to the one given in \cite[2.1.12]{mar}.  For this reason we will introduce several versions of classical first-order notions, relativized to $p$-torsion groups.  We begin by introducing the notion of a complete $p$-torsion pp-type.


\begin{definition} Let $T \supseteq \Th (\Ab)$ be a first-order theory (which is not necessarily complete).  $q(x_1, x_2, \ldots, x_n)$ is a \emph{complete $p$-torsion pp-type consistent with $T$} if there is $M$ a model of $T$ such that $M$ is a $p$-group and $\ba \in M^{n}$ with $q(x_1, x_2, \ldots, x_n) = pp^+(\ba, M)$. 
\end{definition}

In particular, if $q(x_1, x_2, \ldots, x_n)$ is a complete $p$-torsion pp-type consistent with $T$,  $q(x_1, x_2, \ldots, x_n)$ includes a $p$-torsion requirement on each free variable $x_i.$

The next result is the key technical result to prove the main theorem from a model theoretic perspective.

\begin{theorem}\label{T:p-small}
For every $n \geq 1,$ the set of complete $p$-torsion pp-types $q(x_1, x_2, \ldots, x_n)$ consistent with the theory of abelian groups is countable.
\end{theorem}
\begin{proof}
Let us associate to $q(x_1, x_2, \ldots, x_n)$ a triple $(m, \grO, \eta)$ that will only depend on $q(x_1, x_2, \ldots, x_n)$ in such a way that the map
$q \mapsto (m, \grO, \eta)$ is one-to-one and so that the cardinality of the set of such triples is countable. Define the triple $(m, \grO, \eta)$ as follows.
\begin{description}
\item[$m$]  As each $x_i$ satisfies a torsion condition $p^{m_i}x_i \doteq 0 \in q$ for $m_i \in \mathbb{N}$ since the type is realized in a $p$-group,  there is a natural number $m'$ such that $p^{m'} x_i \doteq 0 \in q$ for $i \in \{1, \cdots, n\}.$ Let $m \in \bbN$ be the least such $m'.$ In other words, $m$ is the least natural number 
for which $p^mx_i \doteq 0 \in q$ for $i \in \{1, \cdots, n\}.$
\medskip

\item[$\grO$] Now that $m$ has been fixed, observe that $r_1x_1 + r_2x_2 + \cdots + r_nx_n  \doteq  0 \in q$ if and only if $ r'_1x_1 + r'_2x_2 + \cdots + r'_nx_n  \doteq  0 \in q$  for  $r_i, r_i' \in \mathbb{Z}$ such that $r_i \equiv r'_i \pmod{p^m}$ for $i \in \{1, \cdots, n\}.$ Let 
$$\grO  := \{ (\rbar_1, \rbar_2, \cdots, \rbar_n) \in \bigoplus_{i=1}^n \bbZ (p^m) \; : \;  r_1x_1 + r_2x_2 + \cdots + r_nx_n \doteq 0 \in q \}$$
where $\rbar \in \bbZ (p^m)$ denotes the congruence class modulo $p^m$ of $r \in \bbZ.$ It follows that once $m$ is given, there are only finitely many possibilities for $\grO$ by the finiteness of $\bigoplus_{i=1}^n \bbZ (p^m)$.

\medskip

\item[$\eta$] Now that $m$ has been fixed, observe that  $p^k | r_1x_1 + r_2x_2 + \cdots + r_nx_n  \in q$ if and only if $p^k | r'_1x_1 + r'_2x_2 + \cdots + r'_nx_n \in q,$ for $k, r_i, r_i' \in \mathbb{Z}$  such that $r_i \equiv r'_i \pmod{p^m}$ for $i \in \{1, \cdots, n\}.$

Let  $\eta \colon \bbN \to \mathcal{P}(\bigoplus_{i=1}^n \bbZ (p^m)))$ be given by:

$$\eta(k) := \{ (\rbar_1, \rbar_2, \cdots, \rbar_n) \in \bigoplus_{i=1}^n \bbZ (p^m) \; : \;  p^k|r_1x_1 + r_2x_2 + \cdots + r_nx_n \in q \} \leq \bigoplus_{i=1}^n \bbZ (p^m).$$

Observe that if $k < k^*$ then $\eta(k^*) \subseteq \eta(k)$. Since $\bigoplus_{i=1}^n \bbZ (p^m)$ is finite,  the descending chain $\eta(0) \supseteq  \eta(1) \supseteq \cdots $ is stationary. Hence the map $\eta$ is eventually constant. It follows that once $m$ is given, there are only countably many possibilities for $\eta.$

\end{description}

We have shown that the map $q \mapsto (m, \grO, \eta)$ is well-defined and that its image is countable. To see that it is one-to-one, assume that $(m_q, \grO_q, \eta_q) = (m_s, \grO_s, \eta_s)$ where $q(x_1, x_2, \ldots, x_n), s(x_1, x_2, \ldots, x_n)$ are complete pp-types as given in the statement of the theorem.  By \cite[2.$\mathbb{Z}$.1]{prest} it is enough to show that formulas of the form $r_1x_1 + r_2x_2 + \cdots + r_nx_n  \doteq  0$ and  $p^k | r_1x_1 + r_2x_2 + \cdots + r_nx_n$ are preserved between $q$ and $s$. The first type of formula is preserved by the definition of $\grO$ and the second by that of $\eta$. 
\end{proof}

Let us introduce the notion of $p$-torsion pp-types over non-empty sets.

\begin{definition} Let $M$ be a $p$-group and $A \subseteq M$. $q(x)$ is a \emph{complete $p$-torsion $M$-pp-type over $A$} if there is $N$ an elementary extension of $M$ such that $N$ is a $p$-group and $b \in N$ with $q(x)=\pp^+(b/A, N)$.
\end{definition}

We will be interested in building models which realize a given pp-type and no additional pp-formulas.

\begin{definition}  Let $M$ be a $p$-group and $A \subseteq M$. 
$q(x)$ a complete $p$-torsion $M$-pp-type over $A$ is \emph{strictly realized} in $M^* \succeq M$ if there is $b \in M^{*}$ such that $q(x)=\pp^+(b/A, M^*)$.
\end{definition}

\begin{corollary} \label{C:first homo approx}
If $M$ is a countable $p$-group, then there exists $M_1$ a countable $p$-group elementary extension of  $M$ such that for every finite subset $A$ of $M$ and $q(x)$ a complete $p$-torsion $M$-pp-type over $A$, $q(x)$ is strictly realized in $M_1$. 
\end{corollary}

\begin{proof} We begin by proving the following claim.

\underline{Claim}:  There exists $M'$ a countable elementary extension of  $M$  such that for every finite subset $A$ of $M$ and $q(x)$ a complete $p$-torsion $M$-pp-type over $A$, $q(x)$ is strictly realized in $M'$.

\underline{Proof of Claim:} Let $\{A_i : i < \omega\}$ be an enumeration without repetitions of the finite subsets of $M$. We build $\{ M_n : n < \omega \}$ by induction such that:

\begin{enumerate}
 \item $M_0 = M$.
 \item For every $n$, $M_n$ is countable and $M_n \preceq M_{n+1}$.
 \item For every $n$, $M_{n+1}$ strictly realizes every complete $p$-torsion $M$-pp-types over $A_n$.
\end{enumerate}

The base step is given so we do the induction step.  Let $\{ q_i(x) : i < \omega \}$  be an enumeration without repetitions of the complete $p$-torsion $M$-pp-types over $A_n$, this is possible by Theorem \ref{T:p-small}. We build $\{ N_i : i < \omega \}$ by induction such that:

\begin{enumerate}
 \item $N_0 = M_n$.
 \item For every $i$, $N_i$ is countable and $N_i \preceq N_{i+1}$.
 \item For every $i$, $N_{i+1}$ strictly realizes $q_{i}(x)$.
\end{enumerate} 

The base step is given so we do the induction step. In the induction step, the  standard argument (see for example \cite[4.1.3]{mar}) can be used to obtain $N_{i+1}$ a countable elementary extension of $N_i$ realizing the type $q_i(x) \cup \{ \neg \grf(x) : \grf \text{ is a \emph{pp}-formula and } \grf(x) \notin q_i(x)\} $. Then $N_{i+1}$ strictly realizes $q_{i}(x)$. This finishes the construction of the $N_i$'s. Let $M_{n+1} =  \bigcup_{i < \omega} N_i$. Clearly $M_{n+1}$ is as needed. 

Let $M' = \bigcup_{n < \omega} M_n$. It is clear the $M'$ is as needed. $\dagger_{\text{Claim}}$

Let $M'$ be as above.  To obtain a countable {\em $p$-torsion} elementary extension with the desired property, just take the $p$-torsion subgroup $t_p(M')$ which is a pure subgroup of $M'$.
It follows that for every pair of pp-formulae $\psi \to \grf,$ 
$$|\grf (M)/\psi (M)| \leq |\grf (t_p(M'))/\psi (t_p(M'))| \leq | \grf (M')/\psi (M') | \equiv |\grf (M)/ \psi (M)| \!\! \pmod{\infty}$$
so $t_p(M')$ is elementarily equivalent to $M$ by \cite[2.18]{prest}.  It follows that $M \preceq t_p(M') \preceq M'$ are elementary embeddings by \cite[2.25]{prest}. Moreover, $t_p(M')$ strictly realizes every complete $p$-torsion $M$-pp-type over any  finite subset of $M$ as $M'$ has this property. \end{proof}

We introduce a notion of saturation for $p$-torsion pp-types.

\begin{definition}
\emph{$M$ is $\aleph_0$-saturated for $p$-torsion pp-types} if $M$ is a $p$-group and for every $A$ a finite subset of $M$ and $q(x)$  a complete $p$-torsion $M$-pp-type over $A$, $q(x)$ is strictly realized in $M$. 
\end{definition}


\begin{theorem} \label{T:T-universal torsion}
Let $T \supseteq \Th (\Ab)$ be a complete theory of abelian groups that has a model which is a $p$-group. Then there is $U$ a countable model of $T$ such that $U$  is $\aleph_0$-saturated for $p$-torsion pp-types. Moreover, every countable model of $T$ which is a $p$-group can be elementarily embedded into $U$.
\end{theorem}
\begin{proof}
Let $M \models T$ be a countable $p$-group. Using Corollary \ref{C:first homo approx} it is easy to construct  $\{ M_n : n < \omega \}$ such that:

\begin{enumerate}
\item $M_0 = M$.
\item For every $n$, $M_n$ is a countable $p$-group and $M_n \preceq M_{n+1}$.
\item For every $n$, if $q(x)$ is a complete $p$-torsion $M_{n}$-pp-type over $A$ a finite subset of $M_n$, then $q(x)$ is strictly realized in $M_{n+1}.$
\end{enumerate}

 Let $U = \bigcup_{n<\gro}\, M_n$. $U$ is a countable $p$-group and an elementary extension of $M$ that is $\aleph_0$-saturated for $p$-torsion pp-types by $(3)$.

Let us show that every countable model of $T$ which is a $p$-group can be elementarily embedded into $U$. Let $N$ be a countable model of $T$ which is a $p$-group. Let $\{ n_i : i < \omega \}$ be an enumeration without repetitions of $N$.

We build $\{ f_i : i < \omega \}$ by induction such that:
\begin{enumerate}
\item For every $i$, $f_i: (N, \{ n_k : k \leq i \})  \to (U, \{f_i(n_k) : k \leq i \}) $ is  a partial pure monomorphism.
\item If $i < j$, then $f_i \subseteq f_j$.
\end{enumerate}

Clearly $f = \bigcup_{i < \omega} f_i:  N \to U$ is a pure monomorphism and $f$ is an elementary embedding because $U$ and $N$ are elementary equivalent as they are both models of a complete theory. So it is enough to do the construction.

\underline{Base step}: $i = 0$. Let $p(x) = \pp^+ (n_0,N) \cup \{ \neg \grf(x) : \grf\text{ is a \emph{pp}-formula and } \grf(x) \notin \pp^+ (n_0,N)\}$. Since $T$ is a complete theory, it follows that $p(x)$ is a $U$-type (over $\emptyset$) in the classical sense. So there is $M^*$ an elementary extension of $U$ and $m^* \in M^*$ realizing $p(x)$, see for example \cite[4.1.3]{mar}. An analogous argument to that of Corollary \ref{C:first homo approx} shows that $t_p(M^*)$ is an elementary extension of $U$. Moreover, since $N$ is a $p$-group and $p$-torsion conditions are expressible by \emph{pp}-formulas,  $m^* \in t_p(M^*)$ realizing $p(x)$. Let $q(x)=\pp^+(m^*, t_p(M^*))$. Clearly $q(x)$ is a complete $p$-torsion $U$-pp-type over $\emptyset$, then there is $u_0 \in U$ strictly realizing $q(x)$ because $U$ is $\aleph_0$-saturated for $p$-torsion pp-types. Let $f_0: (N, \{ n_0 \})  \to (U, \{f_0(n_0) \}) $ be given by $f_0(n_0)=u_0$. One can show that $f_0$ is a partial pure monomorphism.

\underline{Induction step}: Let $i = j +1$. Let $p(x) = \pp^+ (n_{j+1}/ \{ n_k : k\leq j\} ,N) \cup  \{ \neg \grf(x, \textbf{n}) : \grf\text{ is a \emph{pp}-formula and } \grf(x, \textbf{n}) \notin \pp^+ (n_{j+1}/ \{ n_k : k\leq j\} ,N)\} $. Let $q(x) = f_{j}(p(x)) = \{ \grf(x, f_j(\textbf{n})) : \grf(x, \textbf{n})   \in p(x) \}  $. Since $T$ is a complete theory, $N, U$ are elementary equivalent. Since $f_j$ is a partial pure monomorphism, by pp-quantifier elimination (see for example \cite[\S 2.4]{prest}) it follows that $q(x)$ is a $U$-type over $\{ f_j(n_k) : k \leq j \}$. Then an analogous argument to the one for the base step can be used to obtain $u_{ j+1} \in U$ such that $f_{j+1} = f_j \cup \{ (n_{j+1}, u_{j+1}) \}$ is a partial pure monomorphism. \end{proof}

\begin{remark}
Let $T = \Th(\bigoplus_{n =1}^\infty \bbZ (p^n)^{(\aleph_0)} \oplus \bbZ (p^\infty)^{(\aleph_0)})$. If $M$ is a $p$-group, it follows from \cite{szm} that $M$ is elementarily equivalent to $\bigoplus_{n =1}^\infty \bbZ (p^n)^{(\alpha_n)} \oplus \bbZ (p^\infty)^{(\alpha_\infty)}$ for $\alpha_n$ countable for every $n \in \mathbb{Z}_{>0} \cup \{ \infty\}$. Then it is easy to show that $M \oplus \bigoplus_{n =1}^\infty \bbZ (p^n)^{(\aleph_0)} \oplus \bbZ (p^\infty)^{(\aleph_0)}$ is a model of $T$. 
\end{remark}
We are ready to provide a positive solution to Fuchs' problem.

\begin{theorem}
There is a countable universal  $p$-group for purity.
\end{theorem}
\begin{proof} Let $T = \Th(\bigoplus_{n =1}^\infty \bbZ (p^n)^{(\aleph_0)} \oplus \bbZ (p^\infty)^{(\aleph_0)})$ and $U$ be the group obtained by applying Theorem \ref{T:T-universal torsion} to $T$. Let $M$ be a countable $p$-group and $f: M \to M \oplus \bigoplus_{n =1}^\infty \bbZ (p^n)^{(\aleph_0)} \oplus \bbZ (p^\infty)^{(\aleph_0)} $ be the inclusion which is clearly a pure monomorphism. Since $M \oplus \bigoplus_{n =1}^\infty \bbZ (p^n)^{(\aleph_0)} \oplus \bbZ (p^\infty)^{(\aleph_0)}$ is a countable model of  $T$, there is $g: M \oplus \bigoplus_{n =1}^\infty \bbZ (p^n)^{(\aleph_0)} \oplus \bbZ (p^\infty)^{(\aleph_0)}  \to U$ an elementary embedding by Theorem \ref{T:T-universal torsion}. Therefore,  $g \circ f: M \to U$ is a pure monomorphism. 
\end{proof}
The existence of a countable universal $p$-group contrasts with the case of reduced $p$-groups, as it was shown in \cite[3.1]{kojsh} that reduced $p$-groups do not have countable universal groups for purity. 
\begin{remark}
Using Lemma \ref{L:summnad} and slightly modifying the construction of Theorem \ref{T:T-universal torsion}, it is possible to obtain the stronger result of Theorem \ref{main:t}. Nevertheless, we do not think that the methods used in this section can be used to obtain the stronger result of Corollary \ref{C: sum-st}.
\end{remark}

\section{Algebraic argument}

In this section we construct a countable  universal $p$-group for purity. The argument given in this section relies heavily on the model theory of modules. For this reason, we present additional preliminaries in this section. The advantage of this argument is that we  explicitly describe a countable universal countable $p$-group for purity and show that for every countable $p$-group there is a pure essential monomorphism  into a direct summand of the mentioned universal group.

\subsection{Additional preliminaries}
The task presents itself to find sufficient conditions for extending a partial homomorphism $f_0 \colon (M,B) \to (N,C)$ to a homomorphism $(f, f_0) \colon (M,B) \to (N,C)$ of groups with parameters. 
One such sufficient condition is when the group $N$ is \emph{pure injective,} that is, if every pure monomorphism $g \colon N \to K$ has a retraction $h \circ g = 1_N.$ 

\begin{fact}[{\cite[3.6]{ziegler}}] \label{F:pi hull}
Given a group $(M,B)$ with parameters, there exists a pure injective group $(H(M, B), B)$ with parameters such that
\begin{enumerate}
\item the morphism $1_B \colon (M,B) \to (H(M, B), B)$ is a partial pure isomorphism;
\item every partial homomorphism $g_0 \colon \colon (M,B) \to (N,C)$ with $N$ pure injective factors through a homomorphism
$$\xymatrix@C=30pt@R=30pt{(M,B) \ar[r]^-{1_B} \ar[dr]_{g_0} & (H(M, B),B) \ar[d]^{(g,g_0)} \\
& (N,C),
}$$
\item every homomorphism $(h, 1_B) \colon (H(M, B), B) \to (H(M, B), B)$ is an isomorphism of groups with parameters.
\end{enumerate}
\end{fact}

\noindent Conditions~(2) and (3) of Fact~\ref{F:pi hull} ensure that the pure injective group $(H(M, B), B)$ with parameters is unique up to isomorphism over $B$; it is called the \emph{pure injective hull} of $(M,B)$ and it is denoted by $H(M,B).$ 
The relationship between $(M,B)$ and its pure injective hull rests on the following key notion.

\begin{definition}
A partial pure monomorphism $f_0 \colon (M,B) \to (N,C)$ is \emph{essential} if every partial homomorphism $h \colon (N,C) \to (K,D)$ whose restriction $h \circ f_0$ to $(M,B)$ is a partial pure monomorphism is itself a partial pure monomorphism,
$$\xymatrix@R=30pt@C=30pt{(M,B) \ar[r]^{f_0} \ar[dr]_{h \circ f_0} & (N,C) \ar[d]^h \\
 & (K,D)
}$$
\end{definition}

Ziegler \cite[\S 3]{ziegler} proved that if $(M,B)$ is a group with parameters and $B \leq C \leq H(M,B),$ then the pure monomorphism $(H(M,B), B) \to (H(M,B), C)$ given by inclusion is an essential pure monomorphism. In particular, the partial pure monomorphism 
$(H(M,B), B) \subseteq (H(M,B), H(M,B))$ is essential, which implies that if $M$ is pure injective then the vertical arrow in the diagram
$$\xymatrix@R=30pt@C=30pt{(M,B) \ar[r]^-{f_0} \ar[dr]_{h \circ f_0} & (H(M,B), H(M,B)) \ar[d]^h \\
 & (M,M)
}$$
is a partial pure monomorphism and therefore a pure monomorphism $h \colon H(M,B) \to M.$ Since $H(M,B)$ is pure injective, we obtain in this way a direct summand of $M,$ which we will denote by $H^M(B),$ isomorphic to $H(M, B)$ over $B.$ If $N$ is a group and $K$ is a pure injective elementary extension of $N$, then $H^K(N)$ is the \emph{pure injective hull of $N$} and we will denote it by $H(N)$ as it does not depend on $K$. 

The following fact follows from the property that for any $a \in H = H(M,B)$ the partial monomorphism $(H, B) \to (H, B + \mathbb{Z}a)$ given by the inclusion $B \leq B + \mathbb{Z}a$ is essential and therefore that the pp-type
$\pp^+ ((B + \mathbb{Z}a)/B, H)$ is maximal, see for example \cite[4.10]{prest}.

\begin{fact} \label{F:linkage}
If $1_B \colon (M,B) \to (H(M, B), B)$ is the pure injective hull of $(M,B)$ and $a \in H(M, B),$ then $B$ and $a$ are linked, i.e., there exists a pp-formula $\grf (u, \bv_\bb)$ such that $H(M) \models \grf (a,B) \wedge \neg \grf (0,B),$ or, equivalently, 
$H(M) \models \grf (a,B) \wedge \neg \grf (a, {\bf 0}).$
\end{fact} 

The next result shows a relationship between essential maps and elementary embeddings. 

 \begin{proposition} \label{P:elem}
A pure monomorphism $f \colon M \to N$ which is essential when considered as a partial pure monomorphism is elementary.
\end{proposition}

\begin{proof} It is enough to show that $M$ and $N$ are elementarily equivalent by \cite[2.25]{prest}. To see that $M$ and $N$ are elementarily equivalent, it is enough to show that they have isomorphic pure injective hulls as pure injective hulls are elementary extensions by \cite{sab}. The composition $g \colon M \to N \to H(N)$ of $f$ with the pure injective hull of $N$ is also essential when considered as a partial homomorphism since $N \to H(N)$ is essential as a partial pure monomorphism.  Since $H(M)$ is pure injective, one can find $h$ such that the following triangle commutes
$$\xymatrix@R=30pt@C=30pt{M \ar[r]^-g \ar[dr]^-d & H(N) \ar[d]^-h \\
 & H(M)}$$
Since $g$ is pure essential, then $h$ is a pure monomorphism. This implies that the image of $h,$ which is isomorphic to $H(N)$ is a direct summand of $H(M)$ that contains $M.$ By minimality of the pure injective hull, it follows that $h$ is onto and therefore an isomorphism.
\end{proof}

An abelian group $M$ is \emph{$\Sigma$-pure injective} if $M^{(\aleph_0)}$ is pure injective. Examples of $\Sigma$-pure injective groups include $\mathbb{Z}(p^n)$ for $n \in \mathbb{Z}_{>0} \cup \{ \infty \}$.  $\Sigma$-pure injective groups are closed under finite direct sums and  under pure subgroups \cite[2.11]{prest}. Moreover, they can be decomposed uniquely as a direct sum of indecomposable pure injective groups \cite[Lemma 1]{gar}.

A group $M$ is \emph{indecomposable} if $M \neq 0$ and $M$ cannot be written as the direct sum of two non-trivial subgroups. If $M$ is an indecomposable pure injective $p$-group, then $M$ is isomorphic to $\mathbb{Z}(p^n)$ with $n \in \mathbb{Z}_{>0} \cup \{ \infty \}$.



\subsection{A countable universal model} If $M$ is a $p$-group, then every finite subset $X \subseteq M$ generates a finite subgroup $B = \langle X \rangle \leq M$, so we will only consider subgroups as parameters in this subsection.
As $B$ is finite, the descending chain of subgroups given by
\begin{equation} 
B = B \cap M \supseteq B \cap pM \supseteq B \cap p^2M \supseteq \cdots \supseteq B \cap p^tM \supseteq \cdots
\end{equation}
is stationary, in the sense that there exists an $s$ such that $B \cap p^sM = B \cap p^tM$ for all $t \geq s.$

The next result uses similar ideas to those of \cite[1.2]{her}. 

\begin{lemma}\label{L:summnad}
Let $M$ be a $p$-group, $B \leq M$ a finite subgroup of parameters and take $s$ to be the least natural number such that for all $t \geq s,$ $B \cap p^tM = B \cap p^sM.$ If $t \geq s,$ then the cyclic group $\bbZ (p^{t +1})$ does not occur as a direct summand of
the pure injective hull $H(M,B).$
\end{lemma}

\begin{proof}
Assume for the sake of contradiction that there exists a direct summand $H(M, B) = \bbZ (p^{t+1}) \oplus H',$ with $t \geq s,$ and choose an element $a \neq 0$ in the socle of $\bbZ (p^{t+1}).$ Thus
$\bbZ (p^{t+1}) \models pa = 0 \wedge p^t \mid a \wedge  p^{t+1} \nmid a.$ 
By Fact~\ref{F:linkage}, $B$ and $a$ are linked, so that there is pp-formula $\grf (u, \bv)$ such that
$$H(M, B) \models \grf (a, B) \wedge \neg \grf (0, B).$$
If $\grf (u,B)$ is expressed as a conjunction, as in Equation~(\ref{Eq:conj}), then one of the conjuncts  $n_j \mid z_j u + \sum_i \; z_{ij}b_i$ already links $a$ to $B.$
Furthermore, as $B$ is a group, $b = - \sum_i \; z_{ij}b_i$ belongs to $B,$ so that $a$ is linked to $B$ by a formula of the form $\grf (u,b) = n \mid zu - b$ for $b \in B$.

If $p \mid z,$ then $H(M, B) \models za = 0,$ which would imply that $H(M, B) \models \grf (0,b),$ contradicting that $\grf$ links $a$ and $B.$ Hence $p \nmid z$. Since $M$ is a $p$-group, this implies that multiplication by $z$ on $M$ is an automorphism that fixes $B$ point-wise. If $zb' = b,$ it follows that the pp-formula 
$\rho (u,b') = n \mid u - b'$ also links $a$ and $B.$ There two cases to consider: \bigskip

\underline{Case 1:} $n = 0$. In this case, the linking formula is equivalent to $u \doteq b',$ which would imply that $a \in B,$ and therefore $a \in B \cap p^sM.$ By the choice of $s,$ $a \in B \cap p^{t+1}M.$ This would imply that $H(M, B) \models p^{t+1} \mid a$. Hence $\bbZ(p^{t+1}) \models p^{t+1} \mid a,$ contradicting that 
$\bbZ(p^{t+1}) \models p^{t+1} \nmid a.$   
\bigskip

\underline{Case 2:} $n \neq 0$. Let $p^k \vert n$ be the highest power of $p$ to divide $n$, so $n= p^kr,$ for some $r$ and $(p,r) = 1.$ Since $B$ is a $p$-group and $a$ has order $p$, one can show that $$ H(M, B) \models p^k \mid a - b' \wedge p^{k}\nmid b',$$
so $\bbZ (p^{t+1}) \models  p^k \nmid a.$ Thus $k > t \geq s$. $H(M, B) \models p^s \mid a - b'$ as $k > s$ and $H(M, B) \models p^s \mid a$ as $t \geq s$, so $H(M, B) \models p^s \mid b'$. Then $M \models p^s \mid b'$,   whence $b' \in B \cap p^sM = B \cap p^kM,$ by the choice of $s.$  This would imply that $M \models p^k \mid b'$, contradicting that $H(M, B) \models p^{k}\nmid b'$.
\end{proof}

\begin{theorem}\label{T:main}
If $M$ is a $p$-group and $B \leq M$ is a finite subgroup of parameters, then $H(M,B)$ is a finite direct sum of groups of the form $\bbZ (p^n)$ for $n \in \mathbb{Z}_{>0} \cup \{ \infty \}.$ In particular $H(M, B)$ is $\Sigma$-pure injective.
\end{theorem}
\begin{proof}
Since $H(M,B)$ is pure injective, it follows that $H(M,B) \cong H( \bigoplus_{i \in  I} M_i)$ such that $M_i$ is a pure injective indecomposable group for every $i \in I$ by \cite[6.1]{ziegler} and \cite[\S 6.3.6]{fuchs}. Moreover, since $M$ is a $p$-group, it follows from Lemma \ref{L:summnad} that  
$$ H(M,B) \cong H( \mathbb{Z}(p)^{(\lambda_1)} \oplus  \mathbb{Z}(p^2)^{(\lambda_2)} \oplus \cdots \oplus  \mathbb{Z}(p^t)^{(\lambda_t)}  \oplus  \mathbb{Z}(p^\infty)^{(\lambda_\infty)} )$$
for $t \in \mathbb{N}$ and $\lambda_i$ a cardinal (possibly finite) for every $i \in \{ 1, \cdots, t, \infty\}$.

Using that pure injective groups are closed under finite direct sums and that $\mathbb{Z}(p^n)$ for $n \in \mathbb{Z}_{>0} \cup \{ \infty \}$ are $\Sigma$-pure injective, it follows that:

$$ H(M,B) \cong \mathbb{Z}(p)^{(\lambda_1)} \oplus  \mathbb{Z}(p^2)^{(\lambda_2)} \oplus \cdots \oplus  \mathbb{Z}(p^t)^{(\lambda_t)}  \oplus \mathbb{Z}(p^\infty)^{(\lambda_\infty)}. $$

Therefore, from the minimality of $H(M, B)$ and the finiteness of  $B$, we have that: 
 
 $$ H(M,B) \cong \mathbb{Z}(p)^{n_1} \oplus  \mathbb{Z}(p^2)^{n_2} \oplus \cdots \oplus  \mathbb{Z}(p^t)^{n_t}  \oplus  Z(p^\infty)^{n_\infty} $$
 for $t \in \mathbb{N}$ and $n_i \in \mathbb{N}$ for every $i \in \{ 1, \cdots, t, \infty\}$.

 Since a finite direct sum of $\Sigma$-pure injectives is $\Sigma$-pure injective, it follows that $H(M, B)$ is $\Sigma$-pure injective.
  \end{proof}

\begin{lemma}\label{L: ess}
Every countable $p$-group admits an essential pure, and therefore elementary, monomorphism into a countable direct sum of groups of the form $\mathbb{Z} (p^n)$ for $n \in \mathbb{Z}_{>0} \cup \{ \infty \}.$
\end{lemma}

\begin{proof}
Let $M$ be a countable $p$-group. Then $M$ is the union of a totally ordered countable set of finite subgroups $$B_0 \subseteq B_1 \subseteq \cdots \subseteq B_n \subseteq \cdots \subseteq M = \bigcup_{n < \omega} \; B_n.$$

One can build $\{ e_n \; \mid \;  n < \omega \}$ such that:

\begin{enumerate}
\item For every $n$, $e_n: H(M, B_n) \to H(M, B_{n+1})$ is a pure monomorphism.
\item For every $n$, the following diagram commutes $$\xymatrix@C=35pt@R=35pt{(M, B_n) \ar[r]^-{\iota_n} \ar[d]^{1_{B_n}} & (M, B_{n+1}) \ar[d]^{1_{B_{n+1}}} \\
(H(M, B_n), B_n) \ar@{>}[r]^-{(e_n, 1_{B_{n+1} \circ \iota_n })} & (H(M,B_{n+1}), B_{n+1})}$$ 
\end{enumerate}

Let $n < \omega$. Since $1_{B_{n+1}} \circ \iota_n: (M, B_n) \to (H(M,B_{n+1}), B_{n+1})$ is a partial homomorphism there is $(e_n, 1_{B_{n+1}} \circ \iota_n)$ making the diagram of (2) above commute by the universal property of the pure injective hull, Fact \ref{F:pi hull}.(2). Moreover, $e_n$ is a pure monomorphism because the composition $1_{B_{n+1}} \circ \iota_n$ is a partial pure monomorphism.

Consider the directed system $( \{H(M,B_n)\}_{n < \omega}, \{ e_{n,m}: H(M,B_n) \to H(M,B_m) \; \mid \;  n \leq m \})$  where $e_{n,m} = 1_{ H(M,B_n)}$ if $m =n$ and $e_{n,m} = e_{m - 1} \circ \cdots \circ e_n$ if $m > n$. Let $( L, \{ \ell_n \}_{n < \omega} )$ be the directed limit. It is known that for every $n < \omega$, $\ell_n :  H(M,B_n) \to L$ is a pure monomorphism and that $L = \bigcup_{n < \omega} \; \ell_n[ H(M,B_n)]$.

Let $h: M \to L$ be given by $h(b)=\ell_n(b)$ if $b \in B_n$. Using the commutativity of the diagram of (2), it is easy to show that $h$ is a well-defined monomorphism.  Moreover, since the $\ell_n$ maps are pure monomorphisms, it follows that $h$ is a pure monomorphism.

 Let us show that $h$ is essential. Let $g \colon (L,L) \to (N,C)$ be a partial homomorphism such that $g \circ h : (M,M) \to (N, C)$ is a partial pure monomorphism.  We show first that for every $n < \omega$, $g \circ \ell_n:  (H(M, B_n), H(M, B_n)) \to (N,C)$ is a partial pure monomorphism. Since $(H(M, B_n), B_n) \to_{1_{B_n}} (H(M, B_n), H(M, B_n))$ is essential, it is enough to show that $g \circ \ell_n \circ 1_{B_n}: (H(M, B_n), B_n) \to (N, C)$ is a partial pure monomorphism. Let $\bb \in B_n$ and $\grf(\bx)$ be a pp-formula. Using first that $1_{B_n}: (M, B_n) \to (H(M, B_n), B_n)$ is a partial pure monomorphism, then  that $g \circ h$ is a partial pure monomorphism and lastly that $h(\bb)= \ell_n(\bb)$, one gets that: 
 
 \[H(M, B_n) \models \grf (\bb) \mbox{ if and only if } M \models \grf (\bb) \mbox{ if and only if } N \models \grf(g\circ h(\bb)) \mbox{ if and only if }  N \models \grf(g\circ \ell_n(\bb)).\]

 Therefore, $g \circ \ell_n :  (H(M, B_n), H(M, B_n)) \to (N,C)$ is a partial pure monomorphism for every $n < \omega$. Since $L = \bigcup_{n < \omega} \; \ell_n[ H(M,B_n)]$ and each $\ell_n$ is a pure monomorphism, it follows that $g \colon (L,L) \to (N,C)$ is a partial pure monomorphism. 

Therefore, the embedding $h \colon M \to L$ is elementary by Proposition~\ref{P:elem}.

To see that the limit $L$ is a direct sum of indecomposable pure injective $p$-groups, let $H'_{n}= \ell_n[ H(M,B_n)]$. For every $n,$ there is a decomposition $H'_{n+1} = H'_n \oplus L_{n +1},$ where $L_{n +1 }$ is a $\Sigma$-pure injective as $H'_{n+1}$ is $\Sigma$-pure injective by Theorem \ref{T:main}. Then it follows from $L = \bigcup_{n < \omega } \; H'_n$ that 
$$L = H'_0 \oplus  \bigoplus_{n=1}^\infty L_n $$
$H'_0$ is a direct sum of $\mathbb{Z} (p^n)$ with $n \in \mathbb{Z}_{>0} \cup \{ \infty \}$ by Theorem \ref{T:main} and $L_i$ is a direct sum of $\mathbb{Z} (p^n),$ with $n \in \mathbb{Z}_{>0} \cup \{ \infty \}$, because $L_i$ is a $\Sigma$-pure injective and hence a direct sum of indecomposable pure injective $p$-groups. Since $L$ is countable we only use countably many direct summands.  
\end{proof}

We are now ready to present the main result of the paper.

\begin{theorem}\label{main:t}
$\bigoplus_{n =1}^\infty \bbZ (p^n)^{(\aleph_0)} \oplus \bbZ (p^\infty)^{(\aleph_0)}$ is a countable universal  $p$-group for purity. 
\end{theorem}
\begin{proof} Let $M$ be a countable $p$-group. Then $M$ is a pure subgroup of a countable direct sum of groups of the form $\mathbb{Z} (p^n)$ for $n \in \mathbb{Z}_{>0} \cup \{ \infty \}$ by Lemma \ref{L: ess}. Hence a pure subgroup of $\bigoplus_{n =1}^\infty \bbZ (p^n)^{(\aleph_0)} \oplus \bbZ (p^\infty)^{(\aleph_0)}$. \end{proof}



 Observe that Lemma \ref{L: ess} can be used to obtain the following stronger result.
 
 \begin{corollary}\label{C: sum-st}
 Every countable $p$-group can be elementarily embedded into a direct summand of $$\bigoplus_{n =1}^\infty \bbZ (p^n)^{(\aleph_0)} \oplus \bbZ (p^\infty)^{(\aleph_0)}$$.
 \end{corollary}

The following result follows directly from the above theorem and the fact that every torsion group can be decomposed into the direct sum of its $p$-subgroups.

\begin{corollary}
$\bigoplus_{p \text{\hspace{0.02cm} prime}} \; \left( \bigoplus_{n =1}^\infty \bbZ (p^n)^{(\aleph_0)} \oplus \bbZ (p^\infty)^{(\aleph_0)}\right)$ is a countable universal torsion group for purity. 
\end{corollary}

The existence of a universal countable torsion abelian group for purity dispels the misconception aired in~\cite[p.\ 297]{sh1151} that a pure monomorphism of torsion abelian groups necessarily preserves the reduced part. The following example provides a concrete illustration of how this may fail to be the case.

\begin{example}\label{E:1}
Let $H_{\omega + 1}$ be the generalized Pr\"{u}fer group generated by $\{a_n : n < \omega \}$ subject to $pa_0 = 0$ and $p^n a_n = a_0$ for $n \geq 1$. Let $G = \bigoplus_{n = 1}^{\infty} \mathbb{Z}(p^n) \oplus \bbZ(p^{\infty})$ where $\bbZ(p^{\infty})$ is generated by  $\{b_n : n < \omega \}$ subject to $pb_0 = 0$ and $pb_{n+1}=b_n$ for $n \geq 0$. 
If $f: H_{\omega + 1} \to G$ is given by $a_0 \mapsto (0, \cdots, b_0)$ and $a_n \mapsto (0, \cdots, 0, 1, 0, \cdots, b_n)$ for $n \geq 1$ where the $1$ is in the $n^{th}$-entry of  $\bigoplus_{n = 1}^{\infty} \mathbb{Z}(p^n)$, then $f$ is a pure monomorphism such that the image of the reduced part of $H_{\omega + 1}$ is not contained in the reduced part of $G$.
\end{example}
\begin{proof}[Proof sketch]
That the image of the reduced part of $H_{\omega + 1}$ is not contained in the reduced part of $G$ is clear as  $H_{\omega + 1}$ is reduced and $f(a_0) \in \bbZ(p^{\infty})$. 

It is easy to see that $f$ is a monomorphism so we show it is pure. Assume that $$G \vDash p^k | (n_1, \cdots, n_m, 0, \cdots, n_0 b_0 + \cdots + n_m b_m)$$ for $n_0, \cdots, n_m \in \mathbb{Z}$ and $k \geq 1$. By switching $n_0$ for some $n^*$ we may assume that $0 \leq n_i < p^i$ for every $1 \leq i \leq m$. For every $1 \leq i \leq m$, since $p^k| n_i$ (mod $p^i$) and $0 \leq n_i < p^i$, it follows that there is $r_i \in \mathbb{N}$ such that $p^kr_i = n_i$. 

Let $y = (r_1, 0,  \cdots, r_1b_1) + (0, r_2, 0, \cdots, r_2b_2) + \cdots (0, \cdots, 0, r_m, 0, \cdots, r_mb_m) + (0, \cdots, 0, n^* , 0, \cdots, n^* b_k) \in f[H_{\omega + 1}]$ where $n^*$ is in the $k^{th}$-entry of $\bigoplus_{n = 1}^{\infty} \mathbb{Z}(p^n)$. It is easy to show that $$G \vDash p^k y =   (n_1, \cdots, n_m, 0, \cdots, n^* b_0 + \cdots + n_m b_m)$$ using that $p^k b_k = b_0$. Therefore, $f$ is a pure monomorphism. \end{proof}




\section{$\aleph_0$-strongly homogeneous $p$-groups}  

In this section we introduce the model theoretic concept of $\aleph_0$-strongly homogeneity for abelian $p$-groups. We use Theorem \ref{T:main} to explicitly describe many $\aleph_0$-strongly homogeneous $p$-groups and to show that the universal model considered in the previous section is $\aleph_0$-strongly homogeneous. Moreover, we give a complete characterization of the countable $\aleph_0$-strongly homogeneous $p$-groups. We finish by showing the relation between $\aleph_0$-strongly homogeneous $p$-groups and transitive $p$-groups.

\begin{definition}
A $p$-group $M$ is \emph{$\aleph_0$-strongly homogeneous}\footnote{Also called $\aleph_0$-sequentially strongly homogeneous in the model theory literature.} if every partial pure isomorphism $f_0 \colon (M,B) \to (M,B')$ for $B, B'$ finite subgroups of $M$ extends to an automorphism $(f, f_0) \colon (M,B) \to (M,B').$\footnote{This definition is equivalent to the standard definition of an $\aleph_0$-strongly homogeneous model by pp-quantifier elimination and the fact that every finite subset generates a finite subgroup in a $p$-group.} 
\end{definition}


The weaker notion of $\aleph_0$-homogeneity, which is equivalent to $\aleph_0$-strongly homogeneity for countable groups, has been studied in some classes of torsion-free groups in \cite{kakh}. Their methods and results are very different to the ones presented in this section.

Using Theorem \ref{T:main} we show that there are many $\aleph_0$-strongly homogeneous $p$-groups.

\begin{lemma} \label{L: countably homo}
If $\{ M_i \mid  i \in I\}$ is a set of indecomposable pure injective $p$-groups, then the direct sum $M = \bigoplus_{i \in I} \; M_i$ is $\aleph_0$-strongly homogeneous.
\end{lemma}
\begin{proof}
Let $B,$ $B' \leq M$ be finite subgroups of parameters and $f_0 \colon (M,B) \to (M,B')$ be a partial pure isomorphism.  Since $B, B'$ are finite, there is a finite set $I_0 \subseteq I,$ such that $B, B' \leq M_{I_0} = \bigoplus_{i \in I_0} \; M_i$.

Then there is $f_1:  H^{M_{I_0}}(B) \cong H^{M_{I_0}}(B')$ with $f_1(b) = f_0(b)$ for every $b \in B$. Moreover, there are $L, L'$ such that $M_{I_0} = H^{M_{I_0}}(B) \oplus L$ and $M_{I_0} = H^{M_{I_0}}(B') \oplus L'$. The group $M_{I_0}$ is $\Sigma$-pure injective, because it is a finite direct sum of $\Sigma$-pure injective groups. Hence $M_{I_0}$ is a direct sum of indecomposable pure injective groups by \cite[Lemma 1]{gar}. Then by the  Krull-Schmidt-Remak-Azumaya Theorem (see for example \cite[2.13]{fac}) the decomposition of $M_{I_0}$ is unique. Since $H^{M_{I_0}}(B), H^{M_{I_0}}(B')$ are isomorphic and the direct sum of finitely many indecomposable pure injectives by Theorem \ref{T:main}, there is $g: L \cong L'$. 

Let $h:= 1_{\bigoplus_{i \notin I_0} M_i} : \bigoplus_{i \notin I_0} \; M_i \cong \bigoplus_{i \notin I_0} \; M_i$. It is clear that $(f_1 \oplus g \oplus h, f_0): (M,B) \to (M,B')$ is as required. \end{proof}





One could use \cite[Theorem 5]{gar}, \cite[4.3]{pipr} instead of Theorem \ref{T:main} to prove the above result, but those results use some deep results of the model theory of modules which we avoid by using Theorem \ref{T:main}. 

We get that the universal model of Theorem \ref{main:t} is $\aleph_0$-strongly homogeneous as $\mathbb{Z}(p^n)$ is an indecomposable pure injective $p$-group for every $n \in \mathbb{Z}_{> 0} \cup \{ \infty \}$.

\begin{corollary}
$\bigoplus_{n =1}^\infty \bbZ (p^n)^{(\aleph_0)} \oplus \bbZ (p^\infty)^{(\aleph_0)}$ is  $\aleph_0$-strongly homogeneous.
\end{corollary}

Moreover, we get the following straightforward corollary which we will extend to arbitrary cardinals in Corollary \ref{C: non}.

\begin{corollary}\label{C : noniso}
There are $2^{\aleph_0}$ countable $\aleph_0$-strongly homogeneous $p$-groups.
\end{corollary}

The next couple of results completely characterize countable $\aleph_0$-strongly homogeneous $p$-groups. Before presenting our results, we recall a few notions from abelian group theory which will play a key role in our results. Let $M$ be a $p$-group. For $\alpha$ an ordinal $p^{\alpha + 1}M = p(p^\alpha M)$ and for $\alpha$ a  limit ordinal $p^{\alpha}M = \bigcap_{\beta < \alpha} p^\beta M$. It is clear that if $f$ is an automorphism of $M$ and $\alpha$ is an ordinal, then $a \in p^{\alpha} M$ if and only if $f(a) \in p^\alpha M$. For $n \in M$, the \emph{height of $n$ in $M$}, denoted by $h_M(n)$, is the ordinal $\alpha$ such that $n \in p^{\alpha}M \backslash p^{\alpha+1}M$ or $\beta$ if $n \neq 0$, $n \in p^{\alpha}M$ for every ordinal $\alpha$ and $\beta$ is the minimum ordinal such that $p^\beta M = p^{\beta + 1} M$ or $\infty$ if $n=0$. The \emph{Ulm sequence} of $n$ in $M$, denoted by $\mathbb{U}_M(n)$, is  $\mathbb{U}_M(n) = (\alpha_1, \alpha_2, \cdots )$ such that $\alpha_i$ is the height of $p^in$ in $M$.

 Let $f_0 \colon (M,B) \to (M,B')$ be a partial pure isomorphism, $f_0$ is \emph{height-preserving} if for every $b \in B$, $h_M(b)=h_M(f_0(b))$. A group $M$ is \emph{reduced} if it does not have a non-trivial divisible subgroup. We will use the following result frequently.
 
 \begin{fact}[{\cite[Exercise 38, Theorem 14]{kapla}}]\label{F: kapla} Assume $M$ is a countable reduced $p$-group. If $B, B'$ are finite subgroups of $M$ and $f_0 \colon (M,B) \to (M,B')$ is a height-preserving partial pure isomorphism, then there  exists an automorphism $f$ of $M$ such that $(f, f_0) \colon (M,B) \to (M,B').$ \end{fact}






We begin by characterizing the non-reduced $\aleph_0$-strongly homogeneous $p$-groups.

\begin{proposition}\label{P: cnr} Assume $M$ is a countable $p$-group that is not reduced. The following are equivalent.
\begin{enumerate}
\item $M$ is $\aleph_0$-strongly homogeneous.
\item $p^{\omega}(M/D) = 0$ where $D$ is the maximal divisible subgroup of $M$
\item $M$ is isomorphic to $ \bigoplus_{n =1}^\infty \bbZ (p^n)^{(\alpha_n)} \oplus \bbZ (p^\infty)^{(\alpha_\infty)}$ such that $\alpha_n$ is countable (possibly finite) for every $n \in \mathbb{Z}_{> 0} \cup \{ \infty \}$. 
\end{enumerate}
\end{proposition}
\begin{proof}
(2) implies (3) follows from \cite[\S 3.5.3]{fuchs} and it is in this step that we use that $M$ is countable, and (3) to (1) follows from Lemma \ref{L: countably homo}. So we only show (1) implies (2).

Assume for the sake of contradiction that $p^{\omega}(M/D) \neq 0$. Since $D$ is divisible there is $R$ reduced such that $M = D \oplus R$ and $R$ is isomorphic to $M/D$. So pick $a \neq 0 \in p^{\omega} R$. Since $R$ is reduced and a direct summand of $M$, there is an ordinal $\alpha$ such that $a \notin p^{\alpha} M$. Let $b \in D$ such that the order of $b$ is equal to the order of $a$, this is possible as $\mathbb{Z}(p^{\infty})$ has elements of all possible $p^n$ orders and $G$ is not reduced. Let $f_0: (M,  \langle a \rangle) \to (M, \langle b \rangle)$ be given by $f(a)=b$. Since $a, b$ have the same order and are in $p^{\omega} M$, it follows from Equation 2 that $f_0: (M,  \langle a \rangle) \to (M, \langle b \rangle)$ is a partial pure isomorphism. Then there is an automorphism $(f, f_0) \colon (M,  \langle a \rangle) \to (M, \langle b \rangle)$. This contradicts the fact that  $a \notin p^{\alpha} M$ and $f(a) = b \in p^{\alpha} M$. Therefore, $p^{\omega}(M/D) = 0$
\end{proof}
\begin{remark}
(1) implies (2) is true for arbitrary cardinals.
\end{remark}

Let $M$ be a $p$-group and $\alpha$ be an ordinal. We will denote by $p^\alpha M[p]$ the elements of $p^\alpha M$ annihilated by $p$. The $\alpha$-th \emph{Ulm invariant} of $M$ is the dimension of  $p^\alpha M[p]/ p^{\alpha + 1} M[p]$ as a vector space over the field of size $p$. Ulm's Theorem (see for example \cite[\S 11]{kapla}) asserts that two reduced countable $p$-groups are isomorphic if and only if they have the same Ulm invariants, making of the Ulm invariants a complete set of invariants for countable $p$-groups. The following result completely characterizes  countable $\aleph_0$-strongly homogeneous reduced $p$-groups by determining the possible Ulm invariants. 

\begin{lemma}\label{L : red}
Assume $M$ is a countable reduced $p$-group. The following are equivalent.
\begin{enumerate}
\item $M$ is $\aleph_0$-strongly homogeneous. 
\item One of the following hold.
\begin{enumerate}
\item $p^\omega M = 0$. In this case, $M$ is isomorphic to $ \bigoplus_{n =1}^\infty \bbZ (p^n)^{(\alpha_n)}$ such that $\alpha_n$ is countable (possibly finite) for every $n \in \mathbb{N}$ and all  the infinite Ulm invariants are zero.
\item There is an $n \geq 1$ such that $p^{\omega + n} M =0$ and $p^\omega M[p] = p^{\omega + (n -1)} M[p] \neq 0$. In this case, there is a unique $n \in \mathbb{N}$ such that the $\omega + (n-1)$-st Ulm invariant is not zero and all the other infinite Ulm invariants are zero.
\end{enumerate}
\end{enumerate}
\end{lemma}
\begin{proof}
$\Rightarrow$:  If $p^\omega M = 0$, $M$ is isomorphic to $ \bigoplus_{n =1}^\infty \bbZ (p^n)^{(\alpha_n)}$ such that all the $\alpha_n$'s are countable (possibly finite) by \cite[\S 3.5.3]{fuchs}.

 Suppose that $p^\omega M \neq 0$. We show first that there  is $n \geq 1$ such that $p^{\omega + n} M = 0$.

\underline{Claim 1}: There exists an $n \geq 1$ such that $p^{\omega + n} M = 0$.

\underline{Proof of Claim 1:} Assume for the sake of contradiction that  $p^{\omega + n} M \neq 0$ for every $n \geq 1$.  Then it follows that for every $n$ there is a $k > n$ such that $p^{\omega + k} M$ is properly contained in $p^{\omega + n} M$, as otherwise $p^{\omega + n} M$ would be a non-zero divisibe subgroup of $M$ contradicting the fact that $M$ is reduced. Therefore, there are $n_1, n_2 \in \mathbb{N}$ and $a, b \in M$ such that  $a \in p^{\omega + n_1} M \backslash p^{\omega + n_1 + n_2} M$ and $b \in  p^{\omega + n_1 + n_2} M$ with the order of $a$ equal to the order of $b$. Since $a, b$ have the same order and are in $p^{\omega} M$, it follows from Equation 2 that $f_0: (M,  \langle a \rangle) \to (M, \langle b \rangle)$ with $f(a)=b$ is a partial pure isomorphism. Then there is an automorphism $(f, f_0) \colon (M,  \langle a \rangle) \to (M, \langle b \rangle)$. This contradicts the fact that  $a \notin  p^{\omega + n_1 + n_2} M$ and $f(a) = b \in  p^{\omega + n_1 + n_2} M$.$\dagger_{\text{Claim 1}}$

Let $n$ be the minimum $n\geq 1$ such that $p^{\omega + n} M = 0$. We show that $p^\omega M[p]  =  p^{\omega + (n -1)} M[p] \neq 0$. It is clear that they are both different from zero, so it is enough to show that $p^\omega M[p] = p^{\omega + (n -1)} M[p]$. Assume for the sake of contradiction that this is not the case, then there is $a \in p^\omega M[p] \backslash p^{\omega + (n -1)} M[p]$. Let $b \neq 0 \in p^{\omega + (n -1)} M[p]$. Since $a$ and $b$ have order $p$ and are in $p^{\omega}M$ a similar argument to that of Claim 1 can be used to obtain a contradiction.

$\Leftarrow$: Let $B,$ $B' \leq M$ be finite subgroups of parameters and $f_0 \colon (M,B) \to (M,B')$ be a partial pure isomorphism. If $p^{\omega} M =0$, it follows that $f_0$ is height-preserving. Then there is an automorphism $(f, f_0) \colon (M,  B) \to (M, B')$ by Fact \ref{F: kapla}. 

So assume that $p^{\omega}M \neq 0$, then there is  $n \geq 1$ such that $p^{\omega + n} M =0$ and \begin{equation} p^\omega M[p] = p^{\omega + (n -1)} M[p] \neq 0 
\end{equation}

The next claim will be the key step of the proof.

\underline{Claim 2}: For every $k \in \mathbb{N}$, for every $a, b \in p^\omega M$, if the order of $a$ is $p^k$ and $g: (M,  \langle a \rangle) \to (M, \langle b \rangle)$ with $g(a)=b$ is a partial pure isomorphism, then $\mathbb{U}_M(a)= \mathbb{U}_M(b)$.

\underline{Proof of Claim 2:}  We do the proof by induction on $k$. 


 If $k=1$, then $h_M(a)=\omega + (n-1) = h_M(b)$ by Equation 5 as $a,b \in p^\omega M[p]$. Clearly $\mathbb{U}_M(a)= \mathbb{U}_M(b)$ since $a, b$ have order $p$.

We do the induction step, let $k > 1$. Observe that $g^*: (M,  \langle pa \rangle) \to (M, \langle pb \rangle)$ with $g^*(pa)=pb$ is a partial pure isomorphism as $g$ is a partial pure isomorphism, so by induction hypothesis $\mathbb{U}_M(pa)= \mathbb{U}_M(pb)$. To finish, it is enough to show that $h_M(a)=h_M(b)$. Assume for the sake of contradiction that $h_M(a)\neq h_M(b)$. We may assume without loss of generality that $h_M(a) < h_M(b) \leq \omega + (n-1)$.

Since $M$ is countable and $\mathbb{U}_M(pa)= \mathbb{U}_M(pb)$, there is $g^{**}$ an automorphism of $M$ such that $g^{**}(pa)=pb$ by \cite[Theorem 24]{kapla}. Then $p(g^{**}(a) - b) =0$ and $g^{**}(a) - b \in p^\omega M[p]$ so $h_M(g^{**}(a) - b) \geq \omega + (n-1)$ by Equation 5. On the other hand, $h_M(g^{**}(a) - b) = min\{ h_M(a), h_M(b)\} < \omega + (n-1)$ by \cite[\S 18]{kapla}, which is clearly a contradiction.$\dagger_{\text{Claim 2}}$

It follows from Claim 2 that $f_0$ is height-preserving. Then there is an automorphism $(f, f_0) \colon (M,  B) \to (M, B')$ by Fact \ref{F: kapla}. \end{proof}

\begin{remark}
(1) implies (2) is true for arbitrary cardinals. We will show in Proposition \ref{P: unc} that if Case (a) holds then $M$ is $\aleph_0$-strongly homogeneous for arbitrary cardinals. 
\end{remark}

It follows from Lemma \ref{L : red} that $H_{\omega + 1}$ (defined in Example \ref{E:1}) is an $\aleph_0$-strongly homogeneous $p$-group with $p^\omega H_{\omega + 1} \neq 0$. $H_{\omega + 1}$ is not a direct sum of indecomposable pure injective $p$-groups and exemplifies that Case (b) of Lemma \ref{L : red} can occur.

The next result puts together Proposition \ref{P: cnr} and Lemma \ref{L : red} and  completely characterizes  countable $\aleph_0$-strongly homogeneous $p$-groups.

\begin{theorem}\label{T: allc} Assume $M$ is a countable $p$-group. The following are equivalent.
\begin{enumerate}
\item $M$ is $\aleph_0$-strongly homogeneous.
\item One of the following holds.
\begin{enumerate}
\item $M$ is isomorphic to $ \bigoplus_{n =1}^\infty \bbZ (p^n)^{(\alpha_n)} \oplus \bbZ (p^\infty)^{(\alpha_\infty)}$ such that $\alpha_i$ is countable (possibly finite) for every $n \in \mathbb{Z}_{>0} \cup \{\infty\}$. 
\item There is an $n \geq 1$ such that $p^{\omega + n} M =0$ and $p^\omega M[p] = p^{\omega + (n -1)} M[p] \neq 0$.
\end{enumerate}
\end{enumerate}
\end{theorem}

We use our results to bound the number of non-isomorphic countable $\aleph_0$-strongly homogeneous  $p$-groups for a complete first-order theory. The result shows that  the main result of \cite{lo} can be extended to a non-elementary setting. We will use frequently that saying that the $n$-th Ulm invariant is of some finite size or $\omega$ is expressible in first-order logic for $n \in \mathbb{N}$.

\begin{lemma}
If $T$ is a complete first-order theory, then there are at most countably many non-isomorphic countable $\aleph_0$-strongly homogeneous $p$-groups\end{lemma}
\begin{proof}
We first bound the number of models of $T$ satisfying condition (a) of Theorem \ref{T: allc}. Every such group is of the form  $\bigoplus_{n =1}^\infty \bbZ (p^n)^{(\alpha_n)} \oplus \bbZ (p^\infty)^{(\alpha_\infty)}$ such that $\alpha_n$ is countable (possibly finite) for every $n \in \mathbb{Z}_{>0} \cup \{\infty\}$. The  $(n-1)$-th Ulm  invariant determines the $\alpha_n$ for $n \in \mathbb{Z}_{>0}$ and these are determined by the first-order theory. Therefore there is only one choice: the value of $\alpha_\infty$. Hence there are at most countably many  $\aleph_0$-strongly homogeneous $p$-groups satisfying condition (a). 

We bound the number of models of $T$ satisfying condition (b) of Theorem \ref{T: allc}. Every such group is determined by the finite Ulm invariants,  the $n$ such that $p^{\omega + n}M = 0$ and the dimension of the $\omega + (n-1)$-st Ulm invariant. Since the finite Ulm invariants are determined by the first-order theory and the other two choices have countably many options, there are at most countably many  $\aleph_0$-strongly homogeneous $p$-groups satisfying condition (b). 
\end{proof}

We turn to uncountable cardinals. The next result due to Baer, which can be found in \cite[\S 10.1.3]{fuchs}, will be useful. If $M$ is a reduced $p$-group and $p^\omega M = 0$, then for every $a \in M$ there is $M_a$ a finite direct summand of $M$ such that $a \in M_a$. 

\begin{proposition}\label{P: unc}
Assume $M$ is a reduced $p$-group. If $p^\omega M = 0$, then $M$ is $\aleph_0$-strongly homogeneous. 
\end{proposition}
\begin{proof}
Let $B,$ $B' \leq M$ be finite subgroups of parameters and $f_0 \colon (M,B) \to (M,B')$ be a partial pure isomorphism. 

\underline{Claim}: If $C$ is a finite subset of $M$, then there is $N$ a finite direct summand of $M$ such that $C \subseteq N$. 

\underline{Proof of Claim:} We do the case when $C = \{ d, e  \}$ as the general case is similar. Applying the result mentioned before the proposition there are $N_1, L_1$ such that $d \in N_1$, $N_1$ is finite and $M = N_1 \oplus L_1$. Then $e = e_1 + e_2$ for $e_1 \in N_1$ and $e_2 \in L_1$. It is clear that $p^\omega L_1 = 0$, so applying the result mentioned before the proposition a second time there are $N_2, L_2$ such that $e_2 \in N_2$, $N_2$ is finite, and $L_1 = N_2 \oplus L_2$. It is straightforward to see that $N =N_1 \oplus N_2$ is as needed.$\dagger_{\text{Claim}}$

Using the above claim we obtain $N, L$ such that $B \cup B' \subseteq N$, $N$ is finite and $M = N \oplus L$. Since $N$ is a direct summand of $M$, $f_0 \colon (N,B) \to (N,B')$ is a partial pure isomorphism and $p^\omega N = 0$. Since $N$ is countable there is an automorphism $(f, f_0)\colon (N,B) \to (N,B')$ by Lemma \ref{L : red}. Let $h:= 1_L : L \cong L$. It is clear that $(f \oplus h, f_0): (M,B) \to (M,B')$ is as required. \end{proof}

We extend Corollary \ref{C : noniso}.

\begin{corollary}\label{C: non}
There are $2^{\lambda}$ non-isomorphic $\aleph_0$-strongly homogeneous $p$-groups of cardinality $\lambda$. 
\end{corollary}
\begin{proof} It follows from Proposition \ref{P: unc} that every reduced $p$-group without elements of infinite height is $\aleph_0$-strongly homogeneous. Then the result follows from \cite[\S 1]{sh74}.
\end{proof}

We use a similar idea to that of Theorem \ref{main:t} to show that every $p$-group is a pure subgroup of a $\aleph_0$-strongly homogeneous $p$-group
\begin{lemma}\label{L :pures}
Every $p$-group is a pure subgroup of an $\aleph_0$-strongly homogeneous $p$-group.
\end{lemma}
\begin{proof}
Let $M$ be a $p$-group and $\lambda \geq |M|$ be such that $\lambda^{\aleph_0}=\lambda$. Let $T = \Th(\bigoplus_{n =1}^\infty \bbZ (p^n)^{(\aleph_0)} \oplus \bbZ (p^\infty)^{(\aleph_0)})$, then $T$ is $\lambda$-stable by \cite[2.1]{ziegler}. Let $N$ be a saturated model of $T$ of cardinality $\lambda$, this exists by \cite[Theorem III.3.12]{shbook}. Let $U=t_p(N)$, it is easy to show that $U$ is of cardinality $\lambda$ and that $U$ is $\aleph_0$-strongly homogeneous. Then doing a similar argument to that of Theorem \ref{main:t}  one can show that there is a pure monomorphism from $M$ into $U$. 
\end{proof}

We finish the paper by showing how $\aleph_0$-strongly homogeneous $p$-groups relate to the classical notion of a transitive group of Kaplansky \cite[\S 18]{kapla}. A $p$-group $M$ is \emph{transitive} if for every $n_1, n_2 \in M$, if $\mathbb{U}_M(n_1) = \mathbb{U}_M(n_2)$ then there is $f$ an automorphism of $M$ sending $n_1$ to $n_2$.\footnote{We do not assume that $M$ is reduced following \cite[2.1]{cago}.}

\begin{proposition}\label{P:h-t}
If $M$ is $\aleph_0$-strongly homogeneous $p$-group, then $M$ is transitive. 
\end{proposition}
\begin{proof}
Let $n_1, n_2 \in M$ and $\mathbb{U}_M(n_1) = \mathbb{U}_M(n_2)$. It is clear that $f_0: (M, \langle n_1 \rangle) \to (M, \langle n_2 \rangle)$ sending $n_1$ to $n_2$ is a partial pure isomorphism by Equation 2. Then there is $f$ an automorphism of $M$ sending $n_1$ to $n_2$. 
\end{proof}

The converse of the above result fails. 

\begin{example} Let $M= H_{\omega +1} \oplus \bbZ (p^\infty)$. Since $M$ is countable, $M$ is transitive by \cite[\S 18]{kapla}. $M$ is not $\aleph_0$-strongly homogeneous by Proposition \ref{P: cnr}.
\end{example}

Observe that the previous example also shows that $\aleph_0$-strongly homogeneous $p$-groups are not closed under direct sums as both  $H_{\omega +1}$ and  $\bbZ (p^\infty)$ are  $\aleph_0$-strongly homogeneous.

The next result, which extends \cite[2.2]{cago}, follows from Proposition \ref{P:h-t} and Lemma \ref{L :pures}.

\begin{corollary}
Every $p$-group is a pure subgroup of a transitive $p$-group.
\end{corollary}



\end{document}